\documentclass[12pt,oneside,english]{amsart}
\usepackage[T1]{fontenc}
\usepackage[latin9]{inputenc}
\usepackage[paperwidth=21cm,paperheight=29cm]{geometry}
\usepackage{units}
\usepackage{amsthm}
\usepackage{amssymb}
\usepackage{setspace}
\usepackage{esint}
\usepackage{bbm}
\makeatletter
\@namedef{subjclassname@2010}{%
	\textup{2010} Mathematics Subject Classification}
\makeatother	

\makeatletter
\usepackage{enumitem}		
\theoremstyle{plain}
\newtheorem{thm}{\protect\theoremname}
  \theoremstyle{plain}
  \newtheorem{lem}[thm]{\protect\lemmaname}
  \theoremstyle{plain}
  \newtheorem{cor}[thm]{\protect\corollaryname}
 \newcommand\thmsname{\protect\theoremname}
 \newcommand\nm@thmtype{theorem}
 \theoremstyle{plain}
 
 \newenvironment{namedthm}[1][Undefined Theorem Name]{
   \ifx{#1}{Undefined Theorem Name}\renewcommand\nm@thmtype{theorem*}
   \else\renewcommand\thmsname{#1}\renewcommand\nm@thmtype{namedtheorem}
   \fi
   \begin{\nm@thmtype}}
   {\end{\nm@thmtype}}

\makeatother

\usepackage{babel}
  \providecommand{\corollaryname}{Corollary}
  \providecommand{\lemmaname}{Lemma}
  \providecommand{\theoremname}{Theorem}
\providecommand{\theoremname}{Theorem}

\begin{document}
\global\long\def\ltnorm#1#2{\left\Vert #1\right\Vert _{L^{2}#2}}
\global\long\def\norm#1{\left\Vert #1\right\Vert }
\global\long\def\bbR{\mathbb{R}}
\global\long\def\bbC{\mathbb{C}}
\global\long\def\bbT{\mathbb{T}}
\global\long\def\bbZ{\mathbb{Z}}
\global\long\def\bbN{\mathbb{N}}
\global\long\def\eps{\varepsilon}
\global\long\def\summ#1#2{\underset{#1}{\overset{#2}{\sum}}}
\global\long\def\rint#1#2{\underset{#1}{\overset{#2}{\int}}}

\let\thefootnote\relax\footnotetext{The research is supported in part by the Israel Science Foundation}
\let\thefootnote\relax\footnotetext{}

\title{Riesz sequences and arithmetic progressions}

\author{Itay Londner and Alexander olevski\u{I}}

\address{School of Mathematical Sciences, Tel-Aviv University, Tel-Aviv 69978,
Israel.}

\email{itaylond@post.tau.ac.il }

\email{olevskii@post.tau.ac.il }
\begin{abstract}
Given a set $\mathcal{S}$ of positive measure on the circle and a
set of integers $\Lambda$, one may consider the family of exponentials
$E\left(\Lambda\right):=\left\{ e^{i\lambda t}\right\} _{\lambda\in\Lambda}$
and ask whether it is a Riesz sequence in the space $L^{2}\left(\mathcal{S}\right)$. 

We focus on this question in connection with some arithmetic properties
of the set of frequencies.

Improving a result of Bownik and Speegle (\cite{MR2270922}, Thm.
4.16), we construct a set $\mathcal{S}$ such that $E\left(\Lambda\right)$
is never a Riesz sequence if $\Lambda$ contains arbitrarily long arithmetic
progressions of length $N$ and step $\ell=O\left(N^{1-\eps}\right)$.
On the other hand, we prove that every set $\mathcal{S}$ admits a
Riesz sequence $E\left(\Lambda\right)$ such that $\Lambda$ does
contain arbitrarily long arithmetic progressions of length $N$ and
step $\ell=O\left(N\right)$.
\end{abstract}

\subjclass[2010]{Primary 42C15; Secondary 42A38}

	\keywords{Riesz sequences, Arithmetic progressions}

\maketitle

\section{Introduction}

We use below the following notation:

$\Lambda$ - A set of integers.

$\mathcal{S}$ - A set of positive measure on the circle $\bbT$.

$\left|\mathcal{S}\right|$ - The Lebesgue measure of $\mathcal{S}$. 

For $A,B\subset\bbR$, $x\in\bbR$ we let 
\[
A+B:=\left\{ \alpha+\beta\,|\,\alpha\in A\,,\,\beta\in B\right\} \;,\; x\cdot A:=\left\{ x\cdot\alpha\,|\,\alpha\in A\right\} \,.
\]

A sequence of elements in a Hilbert space $\left\{ \varphi_{i}\right\} _{i\in I}\subset\mathcal{H}$
is called a $Riesz$ $sequence$ (RS) if there are positive constants
$c,\, C$ s.t. the inequalities 
\begin{align*}
c\underset{i\in I}{\sum}\left|a_{i}\right|^{2} & \leq\left\Vert \underset{i\in I}{\sum}a_{i}\varphi_{i}\right\Vert ^{2}\leq C\underset{i\in I}{\sum}\left|a_{i}\right|^{2}\,,
\end{align*}
hold for every finite sequence of scalars $\left\{ a_{i}\right\} _{i\in I}$.

Given $\Lambda\subset\bbZ$ we denote 
\[
\ensuremath{E\left(\Lambda\right):=\left\{ e^{i\lambda t}\right\} _{\lambda\in\Lambda}}\,.
\]

\begin{doublespace}
The following result is classical (see \cite{MR0107776}, p.203, Lemma
6.5):\\
\emph{If $\Lambda=\left\{ \lambda_{n}\right\} _{n\in\mathbb{N}}\subset\bbZ$
is lacunary in the sense of Hadamard, i.e. satisfies 
\[
\frac{\lambda_{n+1}}{\lambda_{n}}\geq q>1\;,\; n\in\bbN
\]
then $E\left(\Lambda\right)$ forms a RS in $L^{2}\left(S\right)$,
for every $\mathcal{S}\subset\bbT$ with $\left|\mathcal{S}\right|>0$.}

The following generalization is due to I.M. Miheev (\cite{miheev1975lacunary},
Thm. 7):\\
\emph{If $E\left(\Lambda\right)$ is an $S_{p}-system$ for some $p>2$,
i.e. satisfies 
\[
\norm{\underset{\lambda\in\Lambda}{\sum}a_{\lambda}e^{i\lambda t}}_{L^{p}\left(\bbT\right)}\leq C\norm{\underset{\lambda\in\Lambda}{\sum}a_{\lambda}e^{i\lambda t}}_{L^{2}\left(\bbT\right)},
\]
with some $C>0$, for every finite sequence of scalars $\left\{ a_{\lambda}\right\} _{\lambda\in\Lambda}$,}
\emph{then it forms a RS in $L^{2}\left(\mathcal{S}\right)$, for
every $\mathcal{S}\subset\bbT$ with $\left|\mathcal{S}\right|>0$.}

J. Bourgain and L. Tzafriri proved the following result as a consequence
of their ``restricted invertibility theorem'' (\cite{MR890420},
Thm. 2.2):\\
\emph{Given }$\mathcal{S}\subset\bbT$,\emph{ there is a RS $E\left(\Lambda\right)$
s.t. $\Lambda$ is a set of integers with positive asymptotic density
\[
dens\,\Lambda:=\underset{N\rightarrow\infty}{lim}\frac{\#\left\{ \Lambda\cap\left[-N,N\right]\right\} }{2N}>C\left|\mathcal{S}\right|\,.
\]
}(Here and below $C$ denotes positive absolute constants, which might
be different from one another).

W. Lawton (\cite{lawton2010minimal}, Cor. 2.1), \emph{assuming the
Feichtinger conjecture for exponentials}, proved the following proposition:\\
\emph{$\left(L\right)$ For every $\mathcal{S}$ there is a RS $E\left(\Lambda\right)$
s.t. the set of frequencies $\Lambda\subset\bbZ$ is syndetic, that
is $\Lambda+\left\{ 0,\ldots,n-1\right\} =\bbZ$ for some $n\in\bbN$.}
\end{doublespace}

Recall that the Feichtinger conjecture says that every bounded frame in a
Hilbert space can be decomposed in a finite family of RS. This claim
turned out to be equivalent to the Kadison-Singer conjecture (see
\cite{MR2277219}). The last conjecture has been proved recently by
A. Marcus, D. Spielman and N. Srivastava (see \cite{marcus2013interlacing}),
so proposition $\left(L\right)$ holds unconditionally. 

Notice that in some results above the system $E\left(\Lambda\right)$
serves as RS for all sets $\mathcal{S}$; however the set of frequencies
$\Lambda$ is quite sparse there. In others $\Lambda$ is rather dense
but it works for $\mathcal{S}$ given in advance.\\
It was asked in \cite{MR2439002} whether one can somehow combine
the density and $"\mbox{universality}"$ properties. It turned out
this is indeed possible. A set $\Lambda\subset\bbR$ has been constructed
in that paper such that the exponential system $E\left(\Lambda\right)$ forms a RS in $L^{2}\left(\mathcal{S}\right)$
for any open set $\mathcal{S}$ of a given measure, and the set of
frequencies has optimal density, proportional to $\left|\mathcal{S}\right|$.
This is not true for nowhere dense sets (see \cite{MR2439002}).

\section{Results}

In this paper we consider sets of frequencies $\Lambda$ which contain
arbitrarily long arithmetic progressions. Below we denote the length
of a progression by $N$, by $\ell$ we denote its step. Given $\Lambda$
which contains arbitrarily long arithmetic progressions there exists
a set $\mathcal{S}\subset\bbT$ of positive measure so that $E\left(\Lambda\right)$
is not a RS in $L^{2}\left(\mathcal{S}\right)$ (see \cite{miheev1975lacunary}).\\
In the case $\ell$ grows slowly with respect to $N$, one can define
$\mathcal{S}$ independent of $\Lambda$. \\
A quantitative version of such a result was proved in \cite{MR2270922}:\\
\emph{There exists a set $\mathcal{S}$ such that $E\left(\Lambda\right)$
is not a RS in $L^{2}\left(\mathcal{S}\right)$ whenever $\Lambda$ contains arithmetic progressions of length $N_{j}$ ($N_{1}<N_{2}<\ldots$) and the corresponding step $\ell_{j}=o\left(N_{j}^{\nicefrac{1}{2}}log^{-3}N_{j}\right)$.}

The proof is based on some estimates of the discrepancy of sequences
of the form $\left\{ \alpha k\right\} _{k\in\bbN}$ on the circle.

Using a different approach we prove a stronger result:
\begin{thm}
There exists a set $\mathcal{S}\subset\bbT$ such that if a set $\Lambda\subset\bbZ$ contains arithmetic progressions of length $N$(=$N_{1}<N_{2}<N_{3}<\ldots$) with corresponding step  $\ell=O\left(N^{\alpha}\right)$, $\alpha<1$, then $E\left(\Lambda\right)$ is not a RS in $L^{2}\left(\mathcal{S}\right)$.
\end{thm}
Here one can construct $\mathcal{S}$ not depending on $\alpha$ and
with arbitrarily small measure of the complement.

The next theorem shows that the result is sharp.
\begin{thm}
Given a set $\mathcal{S}\subset\bbT$ of positive measure, there is
a set of frequencies $\Lambda\subset\bbZ$ so that:
\begin{enumerate}
\item [(i)] For infinitely many $N$'s $\Lambda$ contains an arithmetic progression of length $N$ and step $\ell=O\left(N\right)$.
\item [(ii)] The system of exponentials $E\left(\Lambda\right)$\emph{
}forms a RS in $L^{2}\left(\mathcal{S}\right)$. 
\end{enumerate}
\end{thm}
Increasing slightly the bound for $\ell$, one can get a version of
Theorem 2 which admits a progression of any length:
\begin{thm}
Given $\mathcal{S}$ one can find $\Lambda$ with the property (ii)
above and s.t. 
\begin{enumerate}
\item [(i')] For every $\alpha>1$ and for every $N\in\bbN$ it contains an arithmetic
progression of length $N$ and step $\ell<C\left(\alpha\right) N^{\alpha}$.
\end{enumerate}
\end{thm}

\section{Proof of theorem 1}
\begin{proof}
Fix $\eps>0$. Take a decreasing sequence of positive numbers $\left\{ \delta\left(\ell\right)\right\} _{\ell\in\bbN}$
s.t. 
\[
\begin{array}{ccc}
\left(a\right) & \summ{\ell\in\bbN}{}\delta\left(\ell\right)<\frac{\eps}{2}\\
\left(b\right) & {\displaystyle \delta\left(\ell\right)\cdot\ell^{\nicefrac{1}{\alpha}}\underset{\ell\rightarrow\infty}{\longrightarrow}\infty} & {\displaystyle \;\forall\alpha\in\left(0,1\right)}
\end{array}
\]

For every $\ell\in\mathbb{N}$ set $I_{\ell}=\left(-\delta\left(\ell\right),\delta\left(\ell\right)\right)$
and let $\tilde{I_{\ell}}$ be the $2\pi$-periodic extension of $I_{\ell}$,
i.e. 
\[
\tilde{I_{\ell}}=\underset{k\in\bbZ}{\bigcup}\left(I_{\ell}+2\pi k\right)\,.
\]
We define 
\begin{equation}
I_{\left[\ell\right]}=\left(\frac{1}{\ell}\cdot\tilde{I_{\ell}}\right)\cap\left[-\pi,\pi\right]\mbox{ and }\mathcal{S}=\bbT\backslash\underset{\ell\in\mathbb{N}}{\bigcup}I_{\left[\ell\right]}=\left(\underset{\ell\in\mathbb{N}}{\bigcup}I_{\left[\ell\right]}\right)^{c}\,,
\end{equation}
whence we get that 
\[
\left|\mathcal{S}\right|\geq1-\underset{\ell=1}{\overset{\infty}{\sum}}\left|I_{\left[\ell\right]}\right|=1-\underset{\ell=1}{\overset{\infty}{\sum}}2\delta\left(\ell\right)>1-\eps\,.
\]
Fix $\alpha<1$ and let $\Lambda\subset\bbZ$ be such that one can
find arbitrarily large $N\in\bbN$ for which 
\[
\left\{ M+\ell,\ldots,M+N\cdot\ell\right\} \subset\Lambda\,,
\]
with some $M=M\left(N\right)\in\bbZ$, $\ell=\ell\left(N\right)\in\bbN$
and 
\begin{equation}
\ell<C\left(\alpha\right) N^{\alpha}\,.
\end{equation}

Recall that by $\left(1\right)$ we have $t\in I_{\left[\ell\right]}$
if and only if $t\ell\in\tilde{I_{\ell}}\cap\left[-\pi\ell,\pi\ell\right]$.
Since $\mathcal{S}$ lies inside the complement of $I_{\left[\ell\right]}$
we get 
\[
\rint{\mathcal{S}}{}\left|\summ{k=1}Nc\left(k\right)e^{i\left(M+k\ell\right)t}\right|^{2}\frac{dt}{2\pi}\leq\rint{I_{\left[\ell\right]}^{^{c}}}{}\left|\summ{k=1}Nc\left(k\right)e^{i\left(M+k\ell\right)t}\right|^{2}\frac{dt}{2\pi}=
\]
\[
=\rint{\left[-\pi\ell,\pi\ell\right]\backslash\tilde{I_{\ell}}}{}\left|\summ{k=1}Nc\left(k\right)e^{ik\tau}\right|^{2}\frac{d\tau}{2\pi\ell}=\rint{I_{\ell}^{^{c}}}{}\left|\summ{k=1}Nc\left(k\right)e^{ik\tau}\right|^{2}\frac{d\tau}{2\pi}\,.
\]
Therefore, in order to complete the proof, it is enough to show that
$\left\Vert \summ{k=1}Nc\left(k\right)e^{ik\tau}\right\Vert _{L^{2}\left(I_{\ell}^{^{c}}\right)}$
can be made arbitrarily small while keeping $\summ{k=1}N\left|c\left(k\right)\right|^{2}$
bounded away from zero. This observation allows us to reformulate
the problem as a norm concentration problem of trigonometric polynomials
of degree $N$ on the interval $I_{\ell}$.

Let $P_{N}\left(t\right)=\frac{1}{\sqrt{N}}\summ{k=1}Ne^{ikt}$,
so $\norm{P_{N}}_{L^{2}\left(\bbT\right)}=1$. Moreover, for every
$t\in\bbT$ we have $\left|P_{N}\left(t\right)\right|\leq\frac{1}{\sqrt{N}\sin\frac{t}{2}}$,
hence 
\[
\underset{I_{\ell}^{^{c}}}{\int}\left|P_{N}\left(t\right)\right|^{2}\frac{dt}{2\pi}\leq\frac{1}{N}\rint{\delta\left(\ell\right)}{\pi}\frac{dt}{\sin^{2}\frac{t}{2}}<\frac{C}{N}\rint{\delta\left(\ell\right)}{\pi}\frac{dt}{t^{2}}<\frac{C}{\delta\left(\ell\right)N}<\frac{C\left(\alpha\right)}{\delta\left(\ell\right)\ell^{\nicefrac{1}{\alpha}}}\,,
\]
where last inequality holds for every $N$ for which $\left(2\right)$
holds. Using condition $\left(b\right)$ we see that indeed last term
can be made arbitrarily small, and so $E\left(\Lambda\right)$ is not
a RS in $L^{2}\left(\mathcal{S}\right)$.

\end{proof}

\section{Proof of theorem 2}

For $n\in\bbN$ we define 
\[
B_{n}:=\left\{ n,\,2n,\ldots,n^{2}\right\} \,.
\]

\begin{lem}
Let $\mathcal{P}$ be the set of all prime numbers. Then the blocks
$\left\{ B_{p}\right\} _{p\in\mathcal{P}}$ are pairwise disjoint.\end{lem}
\begin{proof}
Let $p<q$ be prime numbers. Notice that a number $a\in B_{p}\cap B_{q}$
if and only if there exist $1\leq m\leq p$ and $1\leq k\leq q$ s.t.
\[
a=mp=kq\,,
\]
which is possible only if $q$ divides $m$. But since $m<q$ this
cannot happen and so such $a$ does not exist.\end{proof}
\begin{lem}
Let $\left\{ a\left(n\right)\right\} _{n\in\bbN}$ a sequence of non-negative
numbers s.t. $\underset{n=1}{\overset{\infty}{\sum}}a\left(n\right)\leq1$.
Then for every $\eps>0$ there exist infinitely many $n\in\bbN$ s.t.
\[
\underset{\ell=1}{\overset{n}{\sum}}a\left(\ell n\right)<\frac{\eps}{n}\,.
\]
\end{lem}
\begin{proof}
By Lemma 4 we may write 
\[
\summ{n=1}{\infty}a\left(n\right)\geq\underset{p\in\mathcal{P}}{\sum}\summ{\ell=1}pa\left(\ell p\right)\,.
\]
Assuming the contrary for some $\eps$, i.e. for all but finitely
many $p\in\mbox{\ensuremath{\mathcal{P}}}$ we have $\summ{\ell=1}pa\left(\ell p\right)\geq\frac{\eps}{p}$,
we get a contradiction to the well-known fact that $\underset{p\in\mathcal{P}}{\sum}\frac{1}{p}=\infty$.\end{proof}
\begin{cor}
For every $\eps>0$ there exist infinitely many $n\in\bbN$ s.t. 
\begin{equation}
\underset{\underset{\mu<\lambda}{\lambda,\mu\in B_{n}}}{\sum}a\left(\lambda-\mu\right)<\eps\,.
\end{equation}
\end{cor}
\begin{proof}
Every $\mu<\lambda$ from $B_{n}$ must take the form 
\[
\begin{array}{c}
\lambda=kn\\
\mu=k'n
\end{array},\,1\leq k'<k\leq n,
\]
hence $\lambda-\mu=\ell n$, $\ell\in\left\{ 1,2,\ldots n-1\right\} $.
From Lemma 5 we get for infinitely many $n\in\bbN$

\[
\underset{\underset{\mu<\lambda}{\lambda,\mu\in B_{n}}}{\sum}a\left(\lambda-\mu\right)=\summ{\ell=1}n\left(n-\ell\right) a\left(\ell n\right)\leq n\summ{\ell=1}na\left(\ell n\right)<\eps\,
\]

\end{proof}
Given $B\subset\bbR$, we say that a positive number $\gamma$ is a
lower Riesz bound (in $L^{2}\left(\mathcal{S}\right)$) for the sequence
$E\left(B\right)$ if the inequality 
\[
\norm{\summ{\lambda\in B}{}c\left(\lambda\right)e^{i\lambda t}}_{L^{2}\left(\mathcal{S}\right)}^{2}\geq\gamma\summ{\lambda\in B}{}\left|c\left(\lambda\right)\right|^{2}\,,
\]
holds for every finite sequence of scalars $\left\{ c\left(\lambda\right)\right\} _{\lambda\in B}$.
\begin{lem}
Given $\mathcal{S}\subset\bbT$ of positive measure, there exist a
constant $\gamma=\gamma\left(\mathcal{S}\right)>0$ for which the
following holds: For infinitely many $n\in\bbN$ $\gamma$ is a lower
Riesz bound \emph{(}in $L^{2}\left(\mathcal{S}\right)$\emph{)} for
$E\left(B_{n}\right)$.\end{lem}
\begin{proof}
Let $\mathcal{S}\subset\bbT,$ with $\left|\mathcal{S}\right|>0$.
Applying Corollary 6 to the sequence $\left\{ a\left(n\right)\right\} _{n\in\bbN}:=\left\{ \left|\widehat{\mathbbm{1_{\mathcal{S}}}}\left(n\right)\right|^{2}\right\} _{n\in\bbN}$
(where $\mathbbm{1_{\mathcal{S}}}$ is the indicator function of the
set $\mathcal{S}$), we get for every $\eps>0$ infinitely many $n\in\bbN$
for which $\left(3\right)$ holds. We write
\[
\underset{\mathcal{S}}{\int}\left|\underset{\lambda\in B_{n}}{\sum}c\left(\lambda\right)e^{i\lambda t}\right|^{2}\frac{dt}{2\pi}=\underset{\mathcal{S}}{\int}\left(\underset{\lambda\in B_{n}}{\sum}\left|c\left(\lambda\right)\right|^{2}+\underset{\underset{\lambda\neq\mu}{\lambda,\mu\in B_{n}}}{\sum}c\left(\lambda\right)\overline{c\left(\mu\right)}e^{i\left(\lambda-\mu\right)t}\right)\frac{dt}{2\pi}=
\]
\[
=\left|\mathcal{S}\right|\underset{\lambda\in B_{n}}{\sum}\left|c\left(\lambda\right)\right|^{2}+\underset{\underset{\lambda\neq\mu}{\lambda,\mu\in B_{n}}}{\sum}c\left(\lambda\right)\overline{c\left(\mu\right)}\widehat{\mathbbm{1}_{\mathcal{S}}}\left(\mu-\lambda\right)\,.
\]
By Cauchy-Schwarz inequality, the last term does not exceed
\[
\left|\underset{\underset{\lambda\neq\mu}{\lambda,\mu\in B_{n}}}{\sum}c\left(\lambda\right)\overline{c\left(\mu\right)}\widehat{\mathbbm{1}_{\mathcal{S}}}\left(\mu-\lambda\right)\right|\leq\left(\underset{\lambda,\mu\in B_{n}}{\sum}\left|c\left(\lambda\right)\overline{c\left(\mu\right)}\right|^{2}\right)^{\nicefrac{1}{2}}\left(\underset{\underset{\lambda\neq\mu}{\lambda,\mu\in B_{n}}}{\sum}\left|\widehat{\mathbbm{1}_{\mathcal{S}}}\left(\mu-\lambda\right)\right|^{2}\right)^{\nicefrac{1}{2}}=
\]
\[
=\underset{\lambda\in B_{n}}{\sum}\left|c\left(\lambda\right)\right|^{2}\left(\underset{\underset{\lambda\neq\mu}{\lambda,\mu\in B_{n}}}{\sum}\left|\widehat{\mathbbm{1}_{\mathcal{S}}}\left(\mu-\lambda\right)\right|^{2}\right)^{\nicefrac{1}{2}}.
\]
By $\left(3\right)$ we get 
\[
\underset{\underset{\lambda\neq\mu}{\lambda,\mu\in B_{n}}}{\sum}\left|\widehat{\mathbbm{1}_{\mathcal{S}}}\left(\mu-\lambda\right)\right|^{2}=2\underset{\underset{\mu<\lambda}{\lambda,\mu\in B_{n}}}{\sum}\left|\widehat{\mathbbm{1}_{\mathcal{S}}}\left(\mu-\lambda\right)\right|^{2}<2\eps\,,
\]
 hence 
\[
\underset{\mathcal{S}}{\int}\left|\underset{\lambda\in B_{n}}{\sum}c\left(\lambda\right)e^{i\lambda t}\right|^{2}\frac{dt}{2\pi}\geq\left(\left|\mathcal{S}\right|-\left(2\eps\right)^{\nicefrac{1}{2}}\right)\underset{\lambda\in B_{n}}{\sum}\left|c\left(\lambda\right)\right|^{2}\geq\frac{\left|\mathcal{S}\right|}{2}\underset{\lambda\in B_{n}}{\sum}\left|c\left(\lambda\right)\right|^{2}\,.
\]
Fixing some $\eps<\frac{\left|\mathcal{S}\right|^{2}}{8}$, we get the
last inequality holds for infinitely many $n\in\bbN$.
\end{proof}
The next lemma shows how to combine blocks which correspond to different
progressions.
\begin{lem}
Let $\gamma>0$, $\mathcal{S}\subset\bbT$ with $\left|\mathcal{S}\right|>0$,
and $A_{1},A_{2}\subset\bbN$ finite subsets s.t $\gamma$ is a lower
Riesz bound \emph{(}in $L^{2}\left(\mathcal{S}\right)$\emph{)} for
$E\left(A_{j}\right)$, $j=1,2$. Then for any $0<\gamma'<\gamma$
there exists $M\in\bbZ$ s.t. the system $E\left(A_{1}\cup\left(M+A_{2}\right)\right)$
has $\gamma'$ as a lower Riesz bound.\end{lem}
\begin{proof}
Denote $P_{j}\left(t\right)=\summ{\lambda\in A_{j}}{}c_{j}\left(\lambda\right)e^{i\lambda t}$,
$j=1,2$. Notice that for sufficiently large $M=M\left(\mathcal{S}\right)$,
the polynomials $P_{1}$ and $e^{iMt}P_{2}$ are $"\mbox{almost orthogonal}"$
on $\mathcal{S}$, meaning 
\[
\rint{\mathcal{S}}{}\left|P_{1}\left(t\right)+e^{iMt}\cdot P_{2}\left(t\right)\right|^{2}\frac{dt}{2\pi}=\norm{P_{1}}_{L^{2}\left(\mathcal{S}\right)}^{2}+\norm{P_{2}}_{L^{2}\left(\mathcal{S}\right)}^{2}+o\left(1\right)\,,
\]
where the last term is uniform w.r. to all polynomials having $\norm P_{L^{2}\left(\bbT\right)}=1$.
\end{proof}
Now we are ready to finish the proof of Theorem 2.\\
Given $\mathcal{S}$ take $\gamma$ from Lemma 7 and denote by $\mathcal{N}$
the set of all natural numbers $n$ for which $\gamma$ is a lower
Riesz bound (in $L^{2}\left(\mathcal{S}\right)$) for $E\left(B_{n}\right)$.
Define 
\[
\Lambda=\underset{n\in\mathcal{N}}{\bigcup}\left(M_{n}+B_{n}\right)\,.
\]
Due to Lemma 8 we can define subsequently for every $n\in\mathcal{N}$,
an integer $M_{n}$ s.t. for any partial union 
\[
\Lambda\left(N\right)=\underset{\underset{n<N}{n\in\mathcal{N}}}{\bigcup}\left(M_{n}+B_{n}\right)\,,\, N\in\mathcal{N}
\]
the corresponding exponential system $E\left(\Lambda\left(N\right)\right)$
has lower Riesz bound $\frac{\gamma}{2}\cdot\left(1+\frac{1}{N}\right)$,
so we get that $E\left(\Lambda\right)$ is a RS in $L^{2}\left(\mathcal{S}\right)$.

\section{Proof of theorem 3}

In order to obtain $\Lambda$ which satisfies property \emph{(i')}
we will require the following result.
\begin{namedthm}
[Theorem A] (\cite{apostol1976introduction}, Thm. 13.12) Let $d\left(n\right)$
denote the number of divisors of an integer $n$.\textup{ }\textup{\emph{Then}}\emph{
${\displaystyle d\left(n\right)=o\left(n^{\eps}\right)}$ for every
$\eps>0$.}
\end{namedthm}
The next lemma will be used to control the contribution of blocks when they are not disjoint.
\begin{lem}
Let $\left\{ a\left(n\right)\right\} _{n\in\bbN}$ a sequence of non-negative
numbers s.t. $\underset{n=1}{\overset{\infty}{\sum}}a\left(n\right)\leq1$.
Then for every $\alpha>1$ there exist $\eps\left(\alpha\right)>0$
and $\nu\left(\alpha\right)\in\bbN$ s.t. for every $N\geq\nu\left(\alpha\right)$
one can find an integer $\ell_{\alpha,N}<N^{\alpha}$ satisfying
\begin{equation}
\underset{n=1}{\overset{N}{\sum}}a\left(n\ell_{\alpha,N}\right)<\frac{1}{N^{1+\varepsilon\left(\alpha\right)}}\,.
\end{equation}
\end{lem}
\begin{proof}
Fix $\alpha>1$ and apply Theorem A with $\eps$ small enough, depending
on $\alpha$, to be chosen later. We get the inequality $d\left(k\right)<k^{\varepsilon}$
holds for every $k\geq\nu\left(\alpha\right)$. Fix $N\geq\nu\left(\alpha\right)$,
and notice that for every $L\in\bbN$ 
\[
\underset{\ell=1}{\overset{L}{\sum}}\underset{n=1}{\overset{N}{\sum}}a\left(n\ell\right)\leq\underset{k=1}{\overset{LN}{\sum}}d\left(k\right)a\left(k\right)<\left(LN\right)^{\varepsilon}\,.
\]
It follows that there exists an integer $0<\ell<L$ s.t. 
\[
\underset{n=1}{\overset{N}{\sum}}a\left(n\ell\right)<\frac{\left(LN\right)^{\varepsilon}}{L}=\frac{N^{\varepsilon}}{L^{1-\varepsilon}}\,.
\]
In order to get $\left(4\right)$ we ask 
\[
\frac{N^{\varepsilon}}{L^{1-\varepsilon}}<\frac{1}{N^{1+\varepsilon}}\,,
\]
which yields 
\[
N^{\frac{1+2\eps}{1-\eps}}<L\,,
\]
Therefore, choosing $\eps=\eps\left(\alpha\right)$ sufficiently small
we see that $L$ may be chosen to be smaller than $N^{\alpha}$.
\end{proof}
Setting
\[
B_{\alpha,N}:=\left\{ \ell_{\alpha,N},\,2\ell_{\alpha,N},\ldots,N\ell_{\alpha,N}\right\} \,,
\]
we get
\begin{cor}
Let $\left\{ a\left(n\right)\right\} _{n\in\bbN}$ be as in Lemma 9. For every $\alpha>1$ and $N\geq\nu\left(\alpha\right)$ 
\begin{equation}
\underset{\underset{\mu<\lambda}{\lambda,\mu\in B_{\alpha,N}}}{\sum}a\left(\lambda-\mu\right)<\frac{1}{N^{\varepsilon\left(\alpha\right)}}\,.
\end{equation}

\end{cor}
The proof is identical to that of Corollary 6.

We combine our estimates.
\begin{lem}
Given $\mathcal{S}\subset\bbT$ of positive measure, there exist a
constant $\gamma=\gamma\left(\mathcal{S}\right)>0$ s.t. for every
$\alpha>1$ there exists $N\left(\alpha\right)\in\bbN$ for which
the following holds: For every integer $N\geq N\left(\alpha\right)$
one can find $\ell_{\alpha,N}\in\bbN$ satisfying $\ell_{\alpha,N}<N^{\alpha}$
and $\gamma$ is a lower Riesz bound \emph{(}in $L^{2}\left(\mathcal{S}\right)$\emph{)}
for $E\left(B_{\alpha,N}\right)$.\end{lem}
\begin{proof}
Let $\mathcal{S}\subset\bbT,$ with $\left|\mathcal{S}\right|>0$.
We fix $\alpha>1$ and apply corollary 10 to the sequence $\left\{ a\left(n\right)\right\} _{n\in\bbN}:=\left\{ \left|\widehat{\mathbbm{1_{\mathcal{S}}}}\left(n\right)\right|^{2}\right\} _{n\in\bbN}$,
we get $\eps\left(\alpha\right)$ and for every $N\geq\nu\left(\alpha\right)$
a positive integer $\ell_{\alpha,N}<N^{\alpha}$ satisfying $\left(5\right)$.
Proceeding as in the proof of Lemma 7 we get
\[
\underset{\mathcal{S}}{\int}\left|\underset{\lambda\in B_{\alpha,N}}{\sum}c\left(\lambda\right)e^{i\lambda t}\right|^{2}dt\geq\left(\left|\mathcal{S}\right|-\frac{C}{N^{\nicefrac{\eps\left(\alpha\right)}{2}}}\right)\underset{\lambda\in B_{\alpha,N}}{\sum}\left|c\left(\lambda\right)\right|^{2}\geq\frac{\left|\mathcal{S}\right|}{2}\underset{\lambda\in B_{\alpha,N}}{\sum}\left|c\left(\lambda\right)\right|^{2}\,,
\]
 where last inequality holds for all $N\geq N\left(\alpha\right)$.
\end{proof}
For the last step of the proof we will use a diagonal process.\\
Given $\mathcal{S}$ find $\gamma$ using Lemma 11. This provides,
for every $\alpha>1$ and every $N\geq N\left(\alpha\right)$, a block
$B_{\alpha,N}$ s.t. $\gamma$ is a lower Riesz bound (in $L^{2}\left(\mathcal{S}\right)$)
for $E\left(B_{\alpha,N}\right)$. Let $\alpha_{k}\rightarrow1$ be
a decreasing sequence. Define 
\[
\Lambda=\underset{k\in\bbN}{\bigcup}\underset{N=N\left(\alpha_{k}\right)}{\overset{N\left(\alpha_{k+1}\right)-1}{\bigcup}}\left(M_{N}+B_{\alpha_{k},N}\right)\,.
\]
Again, by Lemma 8, we can make sure any partial union has lower Riesz
bound not smaller than $\frac{\gamma}{2}$, and so $E\left(\Lambda\right)$
is a RS in $L^{2}\left(\mathcal{S}\right)$.

It follows directly from the construction that for every $N\in\bbN$
$\Lambda$ contains an arithmetic progression of length $N$ and step
$\ell<C\left(\alpha\right)N^{\alpha}$, for any $\alpha>1$,
as required.


\end{document}